\documentclass[11pt,smallextended]{amsart}
\usepackage{amsmath,amssymb,amsbsy,amsfonts,amsthm,latexsym,
                        amsopn,amstext,amsxtra,euscript,amscd,bm,stmaryrd}
\usepackage{url}
\usepackage[colorlinks,linkcolor=blue,anchorcolor=blue,citecolor=blue]{hyperref}
\usepackage{color}
\usepackage[english]{babel}
  \usepackage{dsfont}
\usepackage{geometry}
\begin{document}
\newtheorem*{theorem*}{Theorem}
\newtheorem{theorem}{Theorem}
\newtheorem{lemma}[theorem]{Lemma}
\newtheorem{algol}{Algorithm}
\newtheorem{cor}[theorem]{Corollary}
\newtheorem{prop}[theorem]{Proposition}

\newtheorem{proposition}[theorem]{Proposition}
\newtheorem{corollary}[theorem]{Corollary}
\newtheorem*{conjecture*}{Conjecture}
\newtheorem{conjecture}[theorem]{Conjecture}
\newtheorem{definition}[theorem]{Definition}
\newtheorem{remark}[theorem]{Remark}
\renewcommand{\thetheorem}{\empty{}} 
 \numberwithin{equation}{section}
  \numberwithin{theorem}{section}

\newcommand{\comm}[1]{\marginpar{%
\vskip-\baselineskip 
\raggedright\footnotesize
\itshape\hrule\smallskip#1\par\smallskip\hrule}}

\def\sssum{\mathop{\sum\!\sum\!\sum}}
\def\ssum{\mathop{\sum\ldots \sum}}
\def\iint{\mathop{\int\ldots \int}}
\newcommand{\twolinesum}[2]{\sum_{\substack{{\scriptstyle #1}\\
{\scriptstyle #2}}}}

\def\cA{{\mathcal A}}
\def\cB{{\mathcal B}}
\def\cC{{\mathcal C}}
\def\cD{{\mathcal D}}
\def\cE{{\mathcal E}}
\def\cF{{\mathcal F}}
\def\cG{{\mathcal G}}
\def\cH{{\mathcal H}}
\def\cI{{\mathcal I}}
\def\cJ{{\mathcal J}}
\def\cK{{\mathcal K}}
\def\cL{{\mathcal L}}
\def\cM{{\mathcal M}}
\def\cN{{\mathcal N}}
\def\cO{{\mathcal O}}
\def\cP{{\mathcal P}}
\def\cQ{{\mathcal Q}}
\def\cR{{\mathcal R}}
\def\cS{{\mathcal S}}
\def\cT{{\mathcal T}}
\def\cU{{\mathcal U}}
\def\cV{{\mathcal V}}
\def\cW{{\mathcal W}}
\def\cX{{\mathcal X}}
\def\cY{{\mathcal Y}}
\def\cZ{{\mathcal Z}}

\def\C{\mathbb{C}}
\def\F{\mathbb{F}}
\def\K{\mathbb{K}}
\def\Z{\mathbb{Z}}
\def\R{\mathbb{R}}
\def\Q{\mathbb{Q}}
\def\N{\mathbb{N}}
\def\M{\textsf{M}}

\def\({\left(}
\def\){\right)}
\def\[{\left[}
\def\]{\right]}
\def\<{\langle}
\def\>{\rangle}

\title[The maximum size of short character sums.]
{\bf The maximum size of short character sums.}

\author{Marc Munsch}
\address{5010 Institut f\"{u}r  Analysis und Zahlentheorie
8010 Graz, Steyrergasse 30, Graz}
\email{munsch@math.tugraz.at}

\date{\today}


\begin{abstract}
 In the present note, we prove new lower bounds on large values of character sums $\Delta(x,q):=\max_{\chi \neq \chi_0} \vert \sum_{n\leq x} \chi(n)\vert$ in certain ranges of $x$. Employing an implementation of the resonance method developed in a work involving the author in order to exhibit large values of $L$- functions, we improve some results of Hough in the range $\log x = o(\sqrt{\log q})$. Our results are expressed using the counting function of $y$- friable integers less than $x$ where we improve the level of smoothness $y$ for short intervals.
\end{abstract}

\bibliographystyle{plain}
\maketitle

\section{Introduction}

 The behavior of character sums $S_{\chi}(x):=\sum_{n\leq x}\chi(n)$ where $\chi$ is a non-principal Dirichlet character modulo $q$ is of great importance in many number theoretical problems such as the distribution of non-quadratic residues or primitive roots. Showing some cancellation in such character sums has been an intensive topic of study for many decades, originating from the unconditional bound of P\'{o}lya and Vinogradov $S(\chi) \ll \sqrt{q}\log q$. In the present note, we are interested in the opposite problem which consists of bounding from below the quantity 
\begin{equation}\label{max} \Delta(x,q):= \max_{\chi \neq \chi_0}\left\vert \sum_{n \leq x} \chi(n)\right\vert.\end{equation} Here and throughout this paper we write $\log_j$ for the $j$-th iterated logarithm,  for example $\log_2 q = \log \log q$. Assuming the Generalized Riemann Hypothesis, Montgomery and Vaughan showed $\Delta(x,q) \ll \sqrt{q} \log_2 q$ (strenghtening the work of P\'{o}lya and Vinogradov) which matches the omega results obtained earlier by Paley \cite{Paley}. Nonetheless, the situation for shorter intervals remains in certain cases open. An intensive study of this quantity through the computation of high moments was carried out by Granville and Soundararajan \cite{largeGS} giving very precise results for short intervals. Few years later, Soundararajan developed the so-called resonance method \cite{sound} in order to show the existence of large values of $L$- functions at the central point. In the intermediate range $\sqrt{\log q}<\log x<(1-\epsilon) \log q$, even though the situation remains widely open, recent progress have been made using this method (see the results obtained by Hough \cite{Hough} reinforcing the previous results of Granville and Soundararajan \cite{largeGS}). It is worth noticing that de la Bret\`{e}che and Tenenbaum \cite{Gal} recently obtained the following result improving earlier bounds of Hough for very large $x$, precisely as soon as $\log x \geq (\log q)^{1/2+\epsilon}$,
  $$\Delta(x,q) \geq \sqrt{x} \exp\left((\sqrt{2}+o(1))\sqrt{\frac{\log (q/x) \log_3 (q/x)}{\log_2 (q/x)}} \right).$$ \\
  In this paper, we will give new bounds in the range $\log x < \sqrt{\log q}$. 
   Originating from ideas going back to Montgomery and Vaughan \cite{expmulti}, we expect, in that case, the behavior of the character sum to be closely linked to the behavior of the character sum over friable numbers. In particular, it emphases the fact that character sums can only be large because of a special bias for small primes. This fruitful idea can be traced back to Littlewood \cite{Littlewood} which showed the existence of large real character sums by prescribing the values of the first $\log q$ primes. \\ 
 
 In order to state our results, let us define by $\mathcal{S}(x,y)$ the set of $y$- friable numbers less than $x$ and denote by $\Psi(x,y)$ the cardinal of this set. More generally, for any arithmetic function $f$, we write

$$\Psi(x,y;f)=\sum_{n\in \mathcal{S}(x,y)}f(n).$$  Granville and Soundararajan \cite{largeGS} made the following conjecture

\begin{conjecture*} There exists a constant $A>0$ such that for any non-principal character $\chi (\bmod q)$ and for any $1\leq x\leq q$ we have, uniformly,

$$\sum_{n\leq x} \chi(n) = \Psi(x,y;\chi) + o(\Psi(x,y;\chi_0)),$$ where $y=(\log q + \log^2 x)(\log \log q)^A$.
 \end{conjecture*}

This would directly imply the upper bound 
 
 $$\Delta(x,q) \ll \Psi(x,(\log q + \log^2 x)\log_2^A q).$$ 
  On the other hand, this can hold only with $A\geq 1$. Indeed, Hough using the resonance method proved the following estimate in the transition range $\log x=(\log q)^{1/2-\epsilon}$, 
 
 \begin{equation}\label{Hough} \Delta(x,q) \geq \Psi(x,\log q \log_2^{1-o(1)} q).\end{equation}

 In this article, we prove that such a lower bound can be obtained relatively easily in some ranges of $x$.  Our argument relies on the recent variant of the long resonance method introduced in \cite{AMM} and \cite{AMMP}. In few words, the method consists of bounding from below the following quotient
\begin{equation}\label{quotient}
\frac{\left|\sum_{\chi \bmod q} S_{\chi}(x) |R(\chi)^2| \right|}{\sum_{\chi \bmod q} |R(\chi)|^2},
\end{equation}
where $R(\chi)$ is a well-chosen ``resonator''. As explained in \cite{AMM} and \cite{AMMP}, it is possible to define the function $R(\chi)$ as a truncated Euler product of size $\approx \log q\log_2 q$. The main advantage of this method is the complete multiplicative structure of the resonator defined as a short Euler product leading naturally to sums over friable integers. A well suited choice of the weights for every prime allows us to relate the lower bound of the quotient (\ref{quotient}) to the counting function of friable integers with some relatively large level of smoothness. \\
 
 Precisely, we prove

\begin{theorem}\label{lowfriable} For $q$ a sufficiently large prime, under the condition $\log q< x\leq \exp(\sqrt{\log q})$, we have

$$ \Delta(x,q)=\max_{\chi \neq \chi_0} \vert S_{\chi}(x) \vert \geq \Psi \left(x,\left(\frac{1}{4}+o(1)\right)\frac{(\log q)(\log_2 q)}{\max\{(\log_2 x -\log_3 q),\log_3 q\}}\right).   $$

   \end{theorem}

In some ranges of $x$, the lower bound can be rewritten in a more compact way.

\begin{cor}\label{ranges} Suppose that $\log x= (\log q)^{\sigma}$ for a fixed $0<\sigma<1/2$. Then
$$ \Delta(x,q) \geq \Psi\left(x,\frac{1}{2\sigma}(1+o(1))\log q\right).$$
 \end{cor}
 
When $\log x$ is a small power of $\log q$, our result improves the result of Hough which proved (see \cite[Corollary $3.3$]{Hough}) that $\Delta(x,q) \geq \Psi(x,(8/e^3+o(1)) \log q)$. Previously, results of the same quality were proved for real characters. Let us define $\Delta_{\mathbb{R}}(x,q)= \max_{D \in \mathcal{F}, q < \vert D\vert \leq 2q} \vert S_{\chi_D}\vert$ where $\mathcal{F}$ is the set of fundamental discriminants. In the same spirit as Littlewood original argument, Granville and Soundararajan \cite{largeGS} showed $$ \Delta_{\mathbb{R}}(x,q) \geq \Psi\left(x,\frac{1}{3} \log q\right).$$ As pointed out by Hough \cite{Hough}, breaking the natural barrier $\log q$ using this method seems to be a very difficult problem. Corollary \ref{ranges} enables us to do so in the full range $\log x= (\log q)^{\sigma}, 0<\sigma<1/2$. Let us though stretch that our argument can not be adapted to the case of real characters due to the more subtle orthogonality relations. Furthermore, the bound of \cite{Hough} overcomes our result when $\log x$ approaches $\sqrt{\log q}$. \\ 

 For even smaller ranges of $x$, namely as soon as $\log_2 x = o(\log_ 2 q)$, a lower bound of the same precision as in the inequality (\ref{Hough}) follows directly from Theorem \ref{lowfriable}. To illustrate this, we deduce the following consequence 
 

 \begin{cor}\label{smallrange} 
Let $A>1$ and set $x=\log^A q$. Then $$\Delta(x,q) \geq \Psi\left(x,\frac{1}{2}(1+o(1))\frac{(\log q)(\log_2 q)}{\log_3 q}\right).$$
    \end{cor}

 To have a better insight on the lower bounds of Theorem \ref{lowfriable}, Corollaries \ref{ranges} and \ref{smallrange}, we could write it down in a more pleasant way. Indeed (see \cite{HildTensurvey} for the details), in the range $\log q\leq x \leq \exp(\sqrt{\log q})$, the following approximation 
$$\Psi(x,k\log q)=x\exp(-u\log u- u \log_2(u+2)+u+o(u))$$ holds uniformly for $1\leq k\ll x/\log q$ where $u=\frac{\log x}{\log y}$.\\




For the sake of simplicity, we first stated our results for $q$ prime. Howewer, we can obtain similar results for more general moduli $q$ under certain additional restrictions.  In order to describe the result, we define for any integer $m\geq 1$ 

$$\Psi_m(x,y)=\sum_{n\in \mathcal{S}(x,y) \atop (n,m)=1} 1.$$

 Denote by $\omega(q)$ the number of distinct prime divisors of $q$. Tenenbaum \cite{Tencrible} proved that $\Psi_q(x,y) \approx \frac{\phi(q)}{q} \Psi(x,y)$ whenever $\log y \gg (\log 2\omega(q))(\log_2 x)$. Since we have $q/\phi(q)\ll \log(\omega(q))$, we observe that a close argument (setting the resonator with an extra factor $1/(\log(\omega(q))$) leads to the following result

\begin{theorem}\label{lowfriablemod} For $q$ sufficiently large and $\log q \leq x\leq \exp(\sqrt{\log q})$. Under the supplementary condition $\log_2 q \gg \log (1+\omega(q))(\log_2 x - \log_3 q)$,  we have

$$ \Delta(x,q) \geq \Psi_q \left(x,\left(\frac{1}{6}+o(1)\right)\frac{(\log q)(\log_2 q)}{\max\{(\log_2 x-\log_3 q),\log_3 q \}}\right).  $$

   \end{theorem}

   \begin{remark}For almost all modulus $q$, $\omega(q)\sim \log_2 q$ and the restriction on $x$ of Theorem \ref{lowfriablemod} is therefore not so restrictive. It is for instance trivially fulfilled for $x$ being any power of $\log q$. \end{remark} For information, precise estimates of the quantity $\Psi_q(x,y)$ in various ranges of the parameters can be found in \cite{dlBTstat}. \\
  
   In the next section, we give some results concerning averages over friable numbers necessary to control error terms coming from the choice of the resonator. In the last section, we give the proof of the Main Theorem \ref{lowfriable} and sketch the few necessary modifications needed to deduce Corollaries \ref{ranges}, \ref{smallrange} as well as to demonstrate Theorem \ref{lowfriablemod}.


\section{Some results about friables numbers}

As usual, for $x \geq y \geq 2$, we set $u=\frac{\log x}{\log y}$.  We denote by $\Omega(n)$ the number of prime factors of $n$ counted with multiplicity. We easily have the individual bound $\Omega(n) \ll \log n$. In fact, for small values of the parameter $u$, $\Omega(n)$ displays more cancellation when averaged over friable integers.

   \begin{lemma}\label{averageomega}   
 We have, uniformly for $x \geq y \geq 2$,
   $$\sum_{n \in \mathcal{S}(x,y)} \Omega(n) \ll \Psi(x,y) (u + \log_2 y). $$
   \end{lemma}
   
   \begin{proof}

    Applying Theorem $2.9$ of \cite{dlBTstat} to the additive function $f=\Omega$ leads to 
   
   \begin{equation}\label{diviseurs}\sum_{n \in \mathcal{S}(x,y)} \Omega(n) \ll \Psi(x,y) \sum_{p \leq y}\frac{1}{p^{\alpha}-1}\end{equation} where as usual $\alpha:=\alpha(x,y)$ is defined as the unique positive solution of the equation 
   
   $$\sum_{p\leq y} \frac{\log p}{p^{\alpha}-1}= \log x.$$  A precise evaluation of the saddle point $\alpha$ as in \cite[Lemma $3.1$]{dlbTKublius} combined with an application of the Prime Number Theorem 
   gives uniformly 
   $$\sum_{p\leq y} \frac{1}{p^{\alpha}-1} \ll u + \log_2 y$$which concludes the proof. 
   

   \end{proof}
   It is worth noticing that precise asymptotical estimates can be obtain in some ranges of $u$ for averages of general additive functions \cite{dlBTstat}. Moreover, the normal order of additive functions over the set of friable integers $\mathcal{S}(x,y)$ can be determined using the friable Tur\'{a}n-Kubilius inequality  
 originally considered by Alladi \cite{AlladiKubilius} and proved by de la Bret\`{e}che and Tenenbaum \cite{dlbtkubentier} in the full range of parameters $x\geq y \geq 2$. \\
   
   
  
The following lemma gives information on $\Psi(x,y)$ for small variations of $y$. 
   
  \begin{lemma}\label{comparaison}\cite[Corollary $5.6$]{Hough} Assume $u < \sqrt{y}$, then for
$|\kappa| < 1$ we have
\begin{equation}\label{constante}\log \frac{\Psi(x, e^\kappa y)}{\Psi(x,y)} =
\left(\frac{\kappa+ O(\log^{-1}u)}{\log y}\right) u (\log u +
\log_2(u+2)).\end{equation}
\end{lemma}

\section{Resonance method and large values of character sums}

\subsection{Proof of Theorem \ref{lowfriable}}
Denote by $X_q$ the group of characters modulo $q$. For every $\chi \in X_q$ and real $x\geq 2$, we recall the definition 
$$
S_{\chi}(x) = \sum_{n=1}^{x} \chi(n).
$$ 
 In order to exhibit large values of character sums, we apply the resonance method in a similar manner as performed in \cite{AMM} and \cite{AMMP}. Let us first define our ``resonator" as a short Euler product. For a constant $c<1/4$, we take $$y=c\frac{(\log q)(\log_2 q)}{\max\{(\log_2 x-\log_3 q),\log_3 q\}}$$ and set $q_1=1$ and $q_p=0$ for $p>y$. Setting as before $u=\frac{\log x}{\log y}$, we define further $q_p=\left(1-\frac{1}{u(\log_2 q)^{1+\epsilon}}\right)$ for small primes $p\leq y$. We extend it in a completely multiplicative way to obtain weights $q_n$ for all $n\geq 1$. We now define for $\chi\in X_q$ 
$$R(\chi)=\prod_{p\leq y} \left(1-q_p\chi(p)\right)^{-1}.$$ We can write $R(\chi)$ as a Dirichlet series in the form
\begin{equation} \label{rdi}
R(\chi) = \sum_{a=1}^\infty q_a \chi(a).
\end{equation}
Let us consider the following sums 

$$S_1 = \sum_{\chi \in X_q} S_{\chi}(x)\vert R(\chi)\vert^2 $$ and 

$$S_2 =\sum_{\chi \in X_q} \vert R(\chi)\vert^2.$$ The heart of the resonance method is contained in the simple inequality

\begin{equation}\label{resonance} \frac{ \vert S_1 \vert}{S_2 } \leq \max_{\chi \in X_q} \vert S_{\chi}(x)\vert. \end{equation} Therefore we would like to prove that the quotient in the left hand side of (\ref{resonance}) grows sufficiently quickly with $q$. Moreover, we need to show that the contribution of the trivial character to $S_1$ and $S_2$ is in some sense inoffensive. A similar estimation as in \cite{AMMP} leads to 
\begin{equation}\label{ulogq}
\log (|R(\chi)|^2)  \leq  2\sum_{p \leq y} \left( \log u + \log_3 q \right) 
 = 2\frac{y}{\log y}(\log u + \log_3 q+o(1)).
\end{equation}Hence, by our choice of $y$, we have
\begin{equation}\label{Rmaj}
|R(\chi)|^2 \leq \exp((4c+o(1)) \log q) \leq q^{4c + o(1)}.
\end{equation}

 In the other hand, we have trivially, \begin{equation}\label{S2tout} \sum_{\chi \in X_q} \vert R(\chi)\vert^2 =\sum_{\chi\in X_q} q_a q_b \chi(a)\overline{\chi}(b)=\phi(q) \left(\sum_{a=b \bmod q} q_a q_b\right) \geq q^{1-o(1)}\end{equation} where we used the classical inequality $\phi(q) \geq q/\log_2 q$. Similarly, we can disregard the contribution of the trivial character to $S_1$. Indeed, by hypothesis, we have

\begin{equation}\label{trivial} S_{\chi_0}(x) \leq x = q^{o(1)}.\end{equation} Expanding $|R(\chi)|^2 $ and switching the summation, we have 
\begin{eqnarray*}
S_1 & = & \sum_{\chi \in X_q}\vert R(\chi)\vert^2 S_{\chi}(x)\\ 
& = & \sum_{n=1}^{x} \left\{ \sum_{a,b = 1}^\infty q_a q_b 
 \sum_{\chi \in X_q}\chi(a)\overline{\chi}(b)\chi(n)\right\}
\end{eqnarray*} where the inner sum is positive due to the orthogonality relations on $X_q$. Thus,

\begin{equation}\label{quotientdev} S_{1} = \sum_{n=1}^{x}\left\{\phi(q)\sum_{a,b \atop na=b \bmod q} q_a q_b \right\}.\end{equation} Assume $n$ to be fixed such that $(n,q)=1$. Then, using the positivity and the completely multiplicative property of the coefficients $q_n$ we get 

\begin{eqnarray*}
\phi(q)\sum_{a,b\atop na=b \bmod q}  q_a q_b  &\geq& \phi(q)\ \sum_{a,b,  n | b \atop na=b \bmod q}  q_a q_b  \\ & = &\phi(q) \sum_{a,b' \atop na=nb' \bmod q} q_a \underbrace{q_{nb'}}_{=q_n q_{b'}}  = q_n \phi(q) \left(\sum_{a=b' \bmod q} q_a q_{b'}\right) \\
&=& q_n S_2. \\
\end{eqnarray*} Noticing that $q_p=0$ for primes $p>y$, we deduce from (\ref{quotientdev})

\begin{eqnarray}\label{quotientfinal}\frac{S_1}{S_2} &\geq & \sum_{n=1\atop (n,q)=1}^{x} q_n = \sum_{n=1 \atop n\, y- \textrm{friable}}^{x} q_n = \sum_{n\in \mathcal{S}(x,y)} q_n \end{eqnarray} where we used the fact that $q$ is prime. Our particular choice of the weights $q_n$ helps us to relate this sum to the counting function of $y$- friable numbers using results of Section $2$. Indeed, we first remark that for $n\in \mathcal{S}(x,y)$, we have trivially $$q_n \geq \left(1-\frac{1}{u(\log_2 q)^{1+\epsilon}}\right)^{\Omega(n)} \geq 1- \frac{\Omega(n)}{u (\log_2 q)^{1+\epsilon}}.$$ Summing over $n$, we get

\begin{equation}\label{minoration}\sum_{n \in \mathcal{S}(x,y)} q_n \geq \Psi(x,y) + O\left(\frac{1}{u(\log_2 q)^{1+\epsilon}} \sum_{n \in \mathcal{S}(x,y)} \Omega(n) \right).\end{equation}

Using Lemma \ref{averageomega}, we have that 

$$\sum_{n \in \mathcal{S}(x,y)} \Omega(n) \ll \Psi(x,y)(u + \log_2 y).$$


Inserting it in (\ref{minoration}) and combining with (\ref{quotientfinal}), we obtain

\begin{equation}\label{minofinale}\max_{\chi \in X_q} S_{\chi}(x) \geq \Psi(x,y)\left(1+O\left(\frac{\log_3 q}{(\log_2 q)^{1+\epsilon}}\right)\right).\end{equation} Applying Lemma \ref{comparaison}, we deduce after easy computations that the right hand side of (\ref{minofinale}) is bounded from below by $\Psi(x,y(1-\delta)) $ for any $\delta>0$.
Recalling (\ref{Rmaj}), (\ref{S2tout}) and (\ref{trivial}), the inequality (\ref{quotientfinal}) remains true when we sum over non trivial characters yielding to

$$ \max_{\chi \neq \chi_0} S_{\chi}(x) \geq \Psi(x,y(1+o(1))) .$$ 

\subsection{Proof of Corollaries \ref{ranges} and \ref{smallrange}}
In the proof of Theorem \ref{lowfriable}, we can in fact in the appropriate ranges define $y$ in a similar way with $c<1/2$. Indeed in these cases one of the two terms $\log u$ or $\log_3 q$ dominates the other one in Equation (\ref{ulogq}) and thus we can replace $4c$ by $2c+o(1)$ in Equation (\ref{Rmaj}).

\subsection{Proof of Theorem \ref{lowfriablemod}}
Denoting as before $u=\frac{\log x}{\log y}$,  Tenenbaum \cite{Tencrible} proved that $\Psi_q(x,y) \approx \frac{\phi(q)}{q} \Psi(x,y)$ whenever $\log y \gg \log (1+\omega(q))\log (1+u)$. Since we have $q/\phi(q)\ll \log_2 q$, it is sufficient to pertubate slightly the weights in the definition of the resonator. Precisely, for a constant $c<1/6$, we set $$y=c\frac{(\log q)(\log_2 q)}{\max\{(\log_2 x-\log_3 q),\log_3 q\}}$$ and set $q_1=1$ and $q_p=0$ for $p>y$. We define analogously with an extra factor $q_p=\left(1-\frac{1}{u(\log_2 q)^{2+\epsilon}}\right)$ for small primes $p\leq y$. Similarly as in the proof of Theorem \ref{lowfriable}, we get
$$\vert R(\chi_0)\vert \ll q^{6c+o(1)}.$$
After observing that $\displaystyle{\Psi_q(x,y)=\sum_{n\in \mathcal{S}(x,y)} \chi_0(n)}$ where $\chi_0$ denotes the trivial character modulo $q$, the rest of the proof follows exactly the same lines as that of Theorem \ref{lowfriable} using again Lemma \ref{comparaison} to conclude.
\section*{Acknowledgements}

The author would like to thank R\'{e}gis de la Bret\`{e}che for his very helpful comments on an earlier draft. The author is supported by the Austrian Science Fund (FWF) project Y-901 ``Probabilistic methods in analysis and number theory'' led by Christoph Aistleitner.


\begin{thebibliography}{10}

\bibitem{AMM}
C.~Aistleitner, K.~Mahatab, and M.~Munsch.
\newblock {Extreme values of the Riemann zeta function on the 1-line}.
\newblock To appear in IMRN.

\bibitem{AMMP}
C.~Aistleitner, K.~Mahatab, M.~Munsch, and A.~Peyrot.
\newblock {On large values of $L(\sigma,\chi)$}.
\newblock Preprint, available at https://arxiv.org/abs/1803.00760.

\bibitem{AlladiKubilius}
Krishnaswami Alladi.
\newblock {The {T}ur{\'a}n-{K}ubilius inequality for integers without large
  prime factors}.
\newblock {\em J. Reine Angew. Math.}, 335:180--196, 1982.

\bibitem{Gal}
R.~de~la Bret{\`e}che and G.~Tenenbaum.
\newblock {Sommes de G{\`a}l et applications}.
\newblock Preprint, available at https://arxiv.org/abs/1804.01629.

\bibitem{dlbTKublius}
R.~de~la Bret{\`e}che and G.~Tenenbaum.
\newblock {Entiers friables: in{\'e}galit{\'e} de {T}ur{\'a}n-{K}ubilius et
  applications}.
\newblock {\em Invent. Math.}, 159(3):531--588, 2005.

\bibitem{dlBTstat}
R{\'e}gis de~la Bret{\`e}che and G{\'e}rald Tenenbaum.
\newblock {Propri{\'e}t{\'e}s statistiques des entiers friables}.
\newblock {\em Ramanujan J.}, 9(1-2):139--202, 2005.

\bibitem{dlbtkubentier}
R\'egis de~la Bret\`eche and G\'erald Tenenbaum.
\newblock Sur l'in\'egalit\'e de {T}ur\'an--{K}ubilius friable.
\newblock {\em J. Lond. Math. Soc. (2)}, 93(1):175--193, 2016.

\bibitem{largeGS}
Andrew Granville and K.~Soundararajan.
\newblock {Large character sums}.
\newblock {\em J. Amer. Math. Soc.}, 14(2):365--397, 2001.

\bibitem{HildTensurvey}
Adolf Hildebrand and G{\'e}rald Tenenbaum.
\newblock {Integers without large prime factors}.
\newblock {\em J. Th{\'e}or. Nombres Bordeaux}, 5(2):411--484, 1993.

\bibitem{Hough}
Bob Hough.
\newblock {The resonance method for large character sums}.
\newblock {\em Mathematika}, 59(1):87--118, 2013.

\bibitem{Littlewood}
J.~E. Littlewood.
\newblock {On the {C}lass-{N}umber of the {C}orpus {$P({\surd}-k)$}}.
\newblock {\em Proc. London Math. Soc. (2)}, 27(5):358--372, 1928.

\bibitem{expmulti}
H.~L. Montgomery and R.~C. Vaughan.
\newblock {Exponential sums with multiplicative coefficients}.
\newblock {\em Invent. Math.}, 43(1):69--82, 1977.

\bibitem{Paley}
R.~E. A.~C. Paley.
\newblock {A {T}heorem on {C}haracters}.
\newblock {\em J. London Math. Soc.}, 7(1):28--32, 1932.

\bibitem{sound}
K.~Soundararajan.
\newblock {Extreme values of zeta and {$L$}-functions}.
\newblock {\em Math. Ann.}, 342(2):467--486, 2008.

\bibitem{Tencrible}
G.~Tenenbaum.
\newblock {Cribler les entiers sans grand facteur premier}.
\newblock {\em Philos. Trans. Roy. Soc. London Ser. A}, 345(1676):377--384,
  1993.

\end{thebibliography}

\end{document}